\documentclass[11pt]{article}

\usepackage{graphicx}
\usepackage[all]{xy}
\usepackage[colorlinks,citecolor=blue]{hyperref}
\usepackage{amsthm}
\usepackage{amssymb}
\input xy
\xyoption{all}

\newtheorem{theorem}{Theorem}[section]
\newtheorem{lemma}[theorem]{Lemma}

\newtheorem{conjecture}[theorem]{Conjecture}

\begin{document}

\title{The Fundamental Morphism Theorem in the Categories of Graphs \& Graph Reconstruction}

\author{T. Chih
\and
D. Plessas}

%\cortext[cor]{Corresponding author}
%\subjclass{05C60, 05C25}

%\begin{keyword}
%\end{keyword}
\date{\today}

\maketitle

\begin{abstract}
The Fundamental Morphism Theorem is a categorical version of the First Noether Isomorphism Theorem for categories that do not have kernels or cokernels. We consider two categories of graphs. Both categories will admit graphs with multiple edges and loops, and are distinguished by allowing two different types of homomorphisms, standard graph homomorphisms and more general graph homomorphisms where the contraction of an edge is allowed. We establish the Fundamental Morphism Theorem in these two categories of graphs. We then use the result to provide an equivalent reformulation of the vertex and edge reconstruction conjectures. This reformulation shows that reconstructability  is equivalent to the existence of a graph homomorphism satisfying an equation.
\end{abstract}


\section{Introduction}\label{Intro}
We will follow the notations of \cite{Bondy} for graph theory, and in specific we use $\psi$ as the incidence function. The exception to this is that we will name graph homomorphisms as \emph{strict graph morphisms}. We use this terminology to separate strict graph morphisms from a more general graph homomorphism, \emph{graph morphisms}, where edges can be mapped to vertices provided incidence is still preserved. This is not the standard graph homomorphism \cite{HN2004}, but it has two natural advantages. First, it allows the contraction of an edge to be considered as a morphism, and second, it generalizes the morphisms often studied by category theorists when considering the category of directed graphs \cite{Goldblatt, Mac}.\par
The aim of this paper is to establish the Fundamental Morphism Theorem as in \cite{Lawvere} for the category of graphs with graph morphisms and the category of graphs with strict graph morphisms. In both categories we allow for graphs with multiple edges and loops. This result distinguishes these two categories, as the Fundamental Morphism Theorem often fails to hold in non-abelian categories. For example, the Fundamental Morphism Theorem fails to hold in the category of topological spaces and continuous functions. We also show that if you restrict the graphs to be simple in the graph morphism case, or simple with at most one loop allowed on a vertex in the strict graph morphism case, as in the standard category of graphs \cite{HN2004}, the Fundamental Morphism Theorem fails to hold.  \par
Once the Fundamental Morphism Theorem is established, in section \ref{ERC} we apply it to provide a reformulation of the vertex and edge reconstruction conjectures \cite{Harary}. In these reformulations, we only conjecture the existence of an epimorphism that satisfies a single graph homomorphism equation.

\subsection{Categorical Constructions}\label{Prelim}
To aid in a formal definition of a graph morphism, we define the \emph{part set} of a graph $G$ to be $P(G)=E(G)\cup V(G)$. Given two graphs $G$ and $H$, a \emph{graph morphism} $f:G\rightarrow H$ is a function $f_{P}:P(G) \rightarrow P(H)$ with $f_{V}=f_P|_{V(G)}:V(G) \rightarrow V(H)$ that preserves incidence, i.e. $\psi_{H}(f_{P}(e))=\{f_{V}(x),f_{V}(y)\}$ whenever $\psi_{G}(e)=\{x,y\}$, for all $e\in P(G)$ and some $x,y \in V(G)$. As $f_V$ is a restriction of $f_P$, for $p\in P(G)$ we will often write $f(p)$ instead of $f_P(p)$. In this definition of morphism, edges can be mapped to vertices as long as incidence is preserved. If we add the restriction that edges must be mapped to edges, we call the resulting morphism a \emph{strict graph morphism} or strict morphism.\par
We assume the reader is familiar with epimorphisms and momomorphisms. However, for the proof of the result we include the definition for a special type of epimorphism. A morphism $f:A\rightarrow B$ is an \emph{extremal epimorphism} if $f$ does not factor through any proper monomorphism, i.e. if $f=me$ with $m$ a monomorphism and $e$ an epimorphism, then $m$ is an isomorphism \cite{JoyofCats}.\par 
We are concerned with four categories of graphs. We call the category of all graphs with all graph morphisms \textbf{Grphs}, the category of all graphs with strict graph morphisms \textbf{StGrphs}, and the category of simple graphs with all graph morphisms \textbf{SiGrphs}. When we restrict the allowed graphs for the category using strict graph morphisms, we will use simple graphs where at most a single loop is allowed on each vertex. This category will be denoted \textbf{SLStGrphs}, and is the standard category of graphs \cite{HN2004}. K.K. Williams developed versions of the three Noether Isomorphism Theorems for \textbf{Grphs} via a concretely defined quotient graph \cite{KKWil}.\par
We now turn to the required categorical constructions in the four categories of graphs. Proofs that these constructions satisfy the categorical universal mapping properties are straight-forward, and more details can be found in \cite{Plessas}.\par
Given two graphs $A$ and $B$ in \textbf{Grphs}, the categorical product is an generalization of the strong product of graphs where we define $A\times B$ by $V(A\times B)=V(A)\times V(B)$ and for $e\in P(A)$ with $\psi_A(e)=\{a_1,a_2\}$ and $f\in P(B)$ with $\psi_B(f)=\{b_1,b_2\}$ there is an element $(e,f)\in P(A\times B)$ with $\psi_{A\times B}\{(e,f)\}=\{(a_1,b_1),(a_2,b_2)\}$ and if $a_1\neq a_2$ and $b_1\neq b_2$, there is another element $\overline{(e,f)}\in P(A\times B)$ with $\psi_{A\times B}(\overline{(e,f)})=\{(a_1,b_2),(a_2,b_1)\}$ that has the same projections as $(e,f)$. In \textbf{SiGrphs} the categorical product is exactly the strong product. In \textbf{StGrphs} and \textbf{SLStGrphs} the categorical product is the tensor product of graphs, but for our purposes we can follow the construction of \textbf{Grphs} but delete all pairs $(e,f)$ if exactly one of $e$ or $f$ is a vertex.\par
In all four categories of graphs the \emph{coproduct}, $A+B$, of two graphs $A$ and $B$ is the disjoint union of the two graphs, and the \emph{equalizer}, $q=eq(f,g)$, of two morphism $f,g:A\rightarrow B$ is the inclusion morphism of the subgraph $Eq$ of $A$ defined by $P(Eq)=\{ a\in P(A)| f(a)=g(a)$ and if $\psi_A(a)=\{a_1,a_2\}$ then $f(a_1)=g(a_1)$ and $f(a_2)=g(a_2)\}$. The incidence condition ensures that an edge is included in the equalizer only if the incident vertices are as well.\par
In \textbf{Grphs} and \textbf{StGrphs} the \emph{coequalizer}, $coeq(f,g)$, of two morphism $f,g:A\rightarrow B$ is the natural quotient morphism from $B$ to $Coeq$ defined by $P(Coeq)=P(B)/\sim$ where $\sim$ is the equivalence relation defined by $a\sim b$ if there is a sequence $a_0,a_1,\dots,a_n \in P(A)$ such that $a=f(a_0), g(a_0)=f(a_1), g(a_1)=f(a_2), \dots, g(a_{n-1})=f(a_n)$ and $b=f(a_n)$ or $b=g(a_n)$, where if an edge is identified with a vertex, the result is a vertex in $Coeq$.\par
In \textbf{SLStGrphs} we follow the same construction for the coequalizer but we also identify any parallel edges to a single edge and any multiple loops to a single loop, and in \textbf{SiGrphs} we also identify any loops to their incident vertex.\par
In a category with products, coproducts, equalizers, and coequalizers, for a morphism $f:A\rightarrow B$ we can form the following construction,
\begin{equation}\label{eq:FMT} \xymatrix{ R_f\ar[r]^-k & A\times A\ar@<.5ex>[r]^-{p_0}\ar@<-.5ex>[r]_-{p_1} & A \ar[r]^-f\ar[d]_-q & B \ar@<.5ex>[r]^--{i_0}\ar@<-.5ex>[r]_-{i_1} & B+B \ar[r]^-{k^*} & R_f^*\\
 & & I\ar@{-->}[r]_-{\exists !h} & I^*\ar[u]_-{q^*} & & }
\end{equation}\par
where $k=eq(fp_0,fp_1)$, $q=coeq(p_0k,p_1k)$, $k^*=coeq(i_0f,i_1f)$, and $q^*=eq(k^*i_0,k^*i_1)$. This construction yields a unique morphism $h:I\rightarrow I^*$ which makes the diagram commute. We note that $ R_f$ is the kernel pair of $f$ and $R_f^*$ is the cokernel pair of $f$, and we present them in this form to aid in the concrete construction when in the graph categories.\par
The Fundamental Morphism Theorem asserts that $h:I\rightarrow I^*$ is an isomorphism. F.W. Lawvere has shown that the category of sets and functions satisfies the Fundamental Morphism Theorem \cite{Lawvere} which was then extended to the category of discrete topological spaces and continuous functions \cite{Schlomiuk}. The Fundamental Morphism Theorem does not hold in the category of all topological spaces and continuous functions, nor in the category of commutative rings with unit and ring homomorphisms. When the Fundamental Morphism Theorem holds, generalizations of the three Noether Isomorphism Theorems follow as corollaries.\par
We provide two examples of the Fundamental Morphism Theorem  construction (\ref{eq:FMT}) using graphs. In both examples we will consider including a graph of two isolated vertices $\overline{K_2}$ into $K_2$. In Figure \ref{fig:grphs} the construction is formed simultaneously in \textbf{Grphs} and \textbf{StGrphs}. In Figure \ref{fig:sigrphs} we form the same construction simultaneously in \textbf{SiGrphs} and \textbf{SLStGrphs}.  We note that in \textbf{Grphs} and \textbf{StGrphs} the edges in $B+B$ are not identified by $k^*$, and thus an edge is not included in $I^*$. As we add restrictions on the graphs in our categories, the coequalizer morphism identifies parallel edges and loops, and now the edges in $B+B$ are identified by $k^*$ and an edge is included in $I^*$.\par
\begin{figure}[h]
\centering \includegraphics[scale=.6]{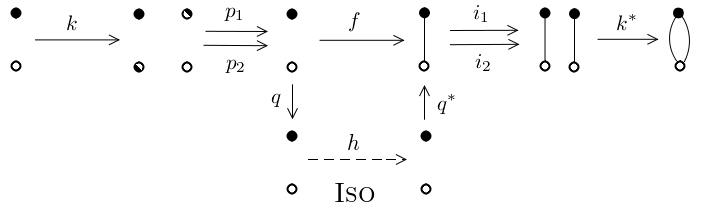}
\caption{An example of the Fundamental Morphism Theorem in \textbf{Grphs} and \textbf{StGrphs} \label{fig:grphs}}
\end{figure} 
\begin{figure}[h]
\centering \includegraphics[scale=.6]{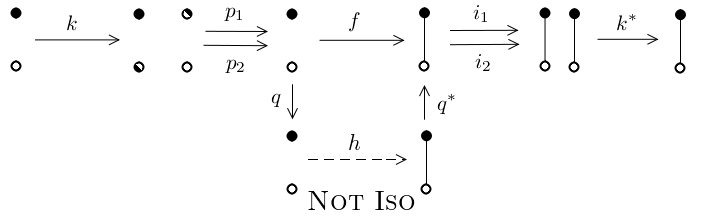}
\caption{A counterexample to the Fundamental Morphism Theorem in \textbf{SiGrphs} and \textbf{SLStGrphs} \label{fig:sigrphs}}
\end{figure}

\section{The Fundamental Morphism Theorem}\label{FMT}
We will establish the Fundamental Morphism Theorem in \textbf{Grphs} and \textbf{StGrphs} but we first need two lemmas concerning the properties of morphisms in these two categories.
\begin{lemma}\label{EpiInj} A morphism of \textbf{Grphs} is an epimorphism if and only if it is a surjective function on part sets, and a morphism of \textbf{Grphs} is a monomorphism if and only if it is an injective function of part sets. The same holds true in \textbf{StGrphs}.
\end{lemma}

\begin{proof}
In concrete categories surjections are always epimorphisms and injections are always monomorphisms. We must only prove the converses. So let $f:A\rightarrow B$ be an epimorphism in \textbf{Grphs}, and suppose $f$ is not surjective on part sets. Then there exists $e\in P(B)\backslash Im(f)$.\par
First suppose $e\in V(B)$. Construct the graph $C$ by appending a vertex $e^\prime$ to $B$ such that $e^\prime$ is adjacent to every vertex $e$ is adjacent to. By construction $B$ is a subgraph of $C$.\par
Now consider $i:B\rightarrow C$ the inclusion morphism and $g:B\rightarrow C$ defined the same as $i$ with the exception that $g(e)=e^\prime$ and for edge $s$ incident to $e$ set $g(s)$ to be the corresponding edge incident to $e^\prime$. This is clearly a morphism (actually it is strict). Then $if=gf$ but $i\neq g$, a contradiction to $f$ being an epimorphism.\par
Now suppose $e$ is an edge of $B$. Construct the graph $C$ by appending an edge $e^\prime$ to $B$ such that $e^\prime$ has the same incidence as $e$. Then by construction $B$ is a subgraph of $C$. Then consider $i:B\rightarrow C$ the inclusion morphism and $g:B\rightarrow C$ to be the same as $i$ except for $g(e)=e^\prime$. Then $g$ is a morphism. Then $if=gf$ but $i\neq g$, a contradiction to $f$ being an epimorphism. Hence epimorphisms in \textbf{Grphs} are surjective functions of the corresponding edge sets.\par
Now let $f:A\rightarrow B$ be an monomorphism in \textbf{Grphs}, and suppose $f$ is not injective. Then there exists $d,e\in P(A)$ such that $f(d)=f(e)$. If either $d$ or $e$ is an edge, then consider $g,h:K_2\rightarrow A$ where $g$ maps the edge to $d$, and the vertices of $K_2$ to the vertices incident to $d$ whereas $h$ maps the edge to $e$ and the vertices of $K_2$ to the vertices incident to $e$. Then as $f$ must preserve incidence, $fg=fh$ but $g\neq h$, a contradiction to $f$ being a monomorphism. If $d$ and $e$ are vertices, a similar contradiction is found for $j,k:K_1\rightarrow A$ where $j$ maps to $d$ and $k$ maps to $e$.\par
A similar proof applies to \textbf{StGrphs}.
\end{proof}
This lemma does not hold in \textbf{SLStGrphs} and in \textbf{SiGrphs}, where it is only required to be injective on the vertex sets to be a monomorphism and surjective on the vertex sets to be an epimorphism. This means that the inclusion of $\overline{K_2}$ into $K_2$ is an epimorphism in these two categories.\par
\begin{lemma}\label{balanced}
Both \textbf{Grphs} and \textbf{StGrphs} are balanced categories, i.e. if a morphism in these categories is an epimorphism and a monomorphism then it is an isomorphism.
\end{lemma}
\begin{proof}
Let $f:A\rightarrow B$ be both a monomorphism and an epimorphism in \textbf{Grphs}. Then by the previous lemma, $f:P(A)\rightarrow P(B)$ is a bijection and there is a set function $f^{-1}:P(B)\rightarrow P(A)$. It suffices to show $f^{-1}$ is a graph morphism.\par
As $f$ is a morphism, $f$ maps vertices to vertices, and as $f$ is a bijection, $f^{-1}$ maps vertices to vertices. Further as monomorphisms are trivially strict morphisms, both $f$ and $f^{-1}$ map edges to edges. Now let $e\in E(B)$ with $\psi_B(e)=\{b_1,b_2\}$ for some $b_1,b_2\in V(B)$, then there is an edge $e^\prime\in E(A)$ with $\psi_A(e^\prime)=\{a_1,a_2\}$ such that $f(e^\prime)=e$. Since $f$ is a morphism, incidence is preserved and $\{b_1,b_2\}=\{f(a_1),f(a_2)\}$. Hence $f^{-1}(e)=e^\prime$ and $\psi_A(f^{-1}(e))=\{f^{-1}(b_1),f^{-1}(b_2)\}=\{a_1,a_2\}=\psi_A(e^\prime)$. Hence incidence is preserved, and $f$ is an isomorphism.\par
As both $f$ and $f^{-1}$ are strict morphisms, the result also holds in \textbf{StGrphs}.
\end{proof}
We note that since \textbf{Grphs} and \textbf{StGrphs} are balanced, all epimorphisms are extremal epimorphisms in these two categories \cite{JoyofCats}. We also note that this lemma fails to hold in $\textbf{SiGrphs}$ and $\textbf{SLStGrphs}$ as the non-isomorphism inclusion of $\overline{K_2}$ into $K_2$ is both an epimorphism and monomorphism.\par
\begin{theorem}
In \textbf{Grphs} and \textbf{StGrphs} the unique morphism $h:I\rightarrow I^*$ in the construction given by (\ref{eq:FMT}) is an isomorphism.
\end{theorem}
\begin{proof}
Consider the construction in \textbf{Grphs}. We first follow part of the standard construction used to prove $h$ exists in order to define a morphism we will require later in the proof.\par
As $q^*=eq(k^*i_0,k^*i_1)$ and $k^*=coeq(i_0f,i_1f)$, $k^*i_0q^*=k^*i_1q^*$ and $k^*i_0f=k^*i_1f$. Then as $q^*$ is an equalizer there is a unique morphism $h^{\prime}:A\rightarrow I^*$ such that the following diagram commutes.
\begin{equation}\label{eq:hprimeprime}
\xymatrix{A \ar[r]^f \ar@{-->}[dr]_{\exists ! h^{\prime}} & B\\
				 & I^* \ar[u]_{q^*}}
\end{equation}\par
We now will prove the claim that $k^*=coeq(i_0f,i_1f)$ identifies parts $i_0(e)$ and $i_1(e)$ for $e\in P(B)$ if and only if $e\in P(Im(f))$.\par
First, let $v\in V(Im(f))$, then there is a vertex $u\in V(A)$ such that $v=f(u)$ for if $v$ is the image of an edge, then $v$ is also the image of the edge's incident vertices. Hence as $i_0f(u)=i_0(v)$ and $i_1f(u)=i_1(v)$, $k^*i_0(v)=k^*i_1(v)$. Then for $e\in E(Im(f))$ there is an edge $e^\prime\in E(A)$ with $f(e^\prime)=e$. Hence as $i_0f(e^\prime)=i_0(e)$ and $i_1f(e^\prime)=i_1(e)$, $k^*i_0(e)=k^*i_1(e)$.\par
We prove the converse holds by contrapositive. Assume that there exists $b\in P(B)\backslash P(Im(f))$. Then for all $a\in P(A)$, $f(a)\neq b$, and hence $i_0f(a)\neq i_0(b)$ and $i_1f(a)\neq i_1(b)$. Thus $i_0(b) \nsim i_1(b)$ and $k^*i_0(b)\neq k^*i_1(b)$ as there is no sequence formed in the construction of the coequalizer between $i_0(b)$ and  $i_1(b)$.\par
We now show that by our definition of equalizer $I^*= Im(f)$.\par 
Let $e\in P(I^*)$, then as $q^*=eq(k^*i_0,k^*i_1)$, $k^*i_0q^*(e)=k^*i_1q^*(e)$. As $q^*$ is inclusion, $k^*i_0(e)=k^*i_1(e)$, and so by our previous claim $e\in P(Im(f))$. Now let $e\in P(Im(f))$. Then by our previous claim, $k^*i_0(e)=k^*i_1(e)$. If $e\in E(Im(f))$ then so are the vertices incident to $e$. As $q^*$ is an equalizer, $e\in P(I)$. Hence as sets $P(I)=P(Im(f))$. Finally, as $q^*$ is a morphism, incidence is preserved and they are equal as graphs.\par
We will now prove the claim that $q=coeq(p_0k,p_1k)$ identifies $a,b\in P(A)$ if and only if $f(a)=f(b)$.\par
We first note that as $p_0((a,b))=a$ and $p_1((a,b))=b$, $P(R_f) = \{(a,b)\in P(A\times A)| f(a)=f(b)$ and if $\psi_{A\times A}((a,b))=\{(u_a,u_b),(v_a,v_b)\}$ then $f(u_a)=f(u_b)$ and $f(v_a)=f(v_b) \}$. So let $a,b\in P(A)$ be such that $q(a)=q(b)$. Then there is a sequence $a_1,a_2,\dots,a_n \in P(R_f)$ with $a=p_0k(a_1),$ $p_1k(a_1)=p_0k(a_2),$ $p_1k(a_2)=p_0k(a_3),$ $\dots,$ $p_1k(a_{n-1})=p_0k(a_n)$ and $b=p_0k(a_n)$ or $b=p_1k(a_n)$. As $k:R_f\rightarrow A\times A$ is inclusion, $a=p_0(a_1),$ $p_1(a_1)=p_0(a_2),$ $p_1(a_2)=p_0(a_3),$ $\dots,$ $p_1(a_{n-1})=p_0(a_n)$ and $b=p_0(a_n)$ or $b=p_1(a_n)$. Then since $a=p_0(a_1)$, $a_1=(a,c_1)$ for some $c_1 \in P(A)$. As $p_0(a_2)=p_1(a_1)$, $a_2=(c_1,c_2)$ for some $c_2\in P(A)$, and inductively $p_0(a_i)=p_1(a_{i-1})$ implies $a_i=(c_{i-1},c_i)$ for $3\leq i\leq n$. Then as $b=p_0(a_n)$ or $b=p_1(a_n)$, $b=c_{n-1}$ or $b=c_n$ respectively. Since $a_1,a_2,\dots,a_n\in P(R_f)$, the object of an equalizer, $f(a)=f(c_1), f(c_1)=f(c_2), \dots, f(c_{n-1})=f(c_n)$ and transitively $f(a)=f(b)$.\par
Conversely, let $a,b\in P(A)$ with $f(a)=f(b)$. We consider two cases.\par
First suppose one of $a$ or $b$ is a vertex or a loop, and without loss of generality, let $a$ be a vertex or a loop. Then $\psi_A(a)=\{u,u\}$ for some $u\in V(A)$. Let $\psi_A(b)=\{u_b,v_b\}$ for some $u_b,v_b\in V(A)$. Since $f(a)=f(b)$ and morphisms preserve incidence, $f(u)=f(u_b)=f(v_b)$. Thus $(a,b)\in P(R_f)$ and as $q=coeq(p_0k,p_1k)$, $p_1k((a,b))=b$ and $p_0k((a,b))=a$, $q(a)=q(b)$.\par
Now consider the case where $a$ and $b$ are non-loop edges. Let $\psi_A(a)=\{u_a,u_b\}$ and $\psi_A(b)=\{v_a,v_b\}$ for some $u_a,u_b,v_a,v_b\in V(A)$. Since $f(a)=f(b)$, $\{f(u_a),f(v_a)\}=\{f(u_b),f(v_b)\}$ and hence either $f(u_a)=f(u_b)$ and $f(v_a)=f(v_b)$ or $f(u_a)=f(v_b)$ and $f(u_b)=f(v_a)$. In the first case $(a,b)\in P(R_f)$ and in the second case $\overline{(a,b)}\in P(R_f)$. As $k$ is inclusion, $p_0((a,b))=p_0(\overline{(a,b)})=a$, and $p_1((a,b))=p_1(\overline{(a,b)})=b$, and $q(a)=q(b)$ as desired.\par
We can now show $h:I\rightarrow I^*$ is a monomorphism. Let $a,b\in V(I)$ with $a\neq b$. As $q$ is a coequalizer, $q$ is an epimorphism and by Lemma \ref{EpiInj} surjective on part sets. Hence there is $u,v\in P(A)$ such that $q(u)=a$ and $q(v)=b$. By the previous claim, as $a\neq b$, $f(u)\neq f(v)$. Then since $f=q^*hq$ and $q^*$ is inclusion, $h(a)=q^*h(a)=q^*hq(u)=f(u)\neq f(v)=q^*hq(v)=q^*h(b)=h(b)$, and $h$ is an injection on part sets. Then by Lemma \ref{EpiInj}, $h$ is a monomorphism.\par
We now show $h^{\prime}$ defined as in (\ref{eq:hprimeprime}) is an epimorphism (and by Lemma \ref{balanced} an extremal epimorphism).\par
Define $\overline{h}:A \rightarrow I^*$ by $\overline{h}(e)=f(e)$ for all $e\in P(A)$. As $Im(f)=I^*$ and $f$ is a morphism, $\overline{h}$ is well defined and a morphism. Since $q^*$ is inclusion, $q^*\overline{h}(a)=q^*(f(a))=f(a)$ for all $a\in P(A)$. Thus $q^*\overline{h}=f$. However, $h^{\prime}$ is the unique morphism such that $q^*h^{\prime}=f$. Therefore $\overline{h}=h^{\prime}$. As $\overline{h}$ is a surjection on part sets, so is $h^{\prime}$. Thus by Lemma \ref{EpiInj} $h^{\prime}$ is an epimorphism.\par
Finally, as $h^{\prime}=hq$ is a factorization of an extremal epimorphism through a monomorphism, $h$ is an isomorphism.\par
The proof for \textbf{StGrphs} follows similarly.
\end{proof}

\section{A Reformulation of Reconstruction Conjectures}\label{ERC}
We can leverage the Fundamental Morphism Theorem to establish isomorphisms in the same way the Noether Isomorphism Theorems are used in Algebra. We provide an example of this by giving a reformulation the Edge Reconstruction Conjecture of Harary \cite{Harary} in terms of establishing the existence of a morphism that satisfies a certain equation.\par
We first recall the Edge Reconstruction Conjecture in the form of isomorphic edge-decks. 
\begin{conjecture}[Edge Reconstruction Conjecture]
Let $G$ and $H$ be finite simple graphs with at least $4$ edges and single edge-deleted subgraphs $G_i$ and $H_i$ for $i\in I$ for some indexing set $I$, where there are isomorphisms $\gamma_i:G_i \to H_i$ for all $i\in I$, then $G$ is isomorphic to $H$.
\end{conjecture}
As one of the most famous conjectures in graph theory, this conjecture has generated many remarkable results \cite{bollobas, muller}. Fantastic survey papers on the many approaches to the conjecture are available for the interested reader \cite{asciak, bondysurvey, lauri, maccari}.\par
Given graphs $G$ and $H$ as in our statement of the Edge Reconstruction Conjecture, we will apply the Fundamental Morphism Theorem from $\mathbf{StGrphs}$. Let $\displaystyle \coprod_{i\in I}G_i$ be the coproduct of the all of the edge-deleted subgraphs of $G$. As $H_i$ has canonical inclusion morphisms $\kappa_i:H_i\to H$ for all $i\in I$, there are morphisms $\kappa_i\gamma_i:G_i\to H$ for all $i\in I$. So by the universal mapping property of the coproduct, there is a morphism $\displaystyle \Gamma: \coprod_{i\in I}G_i\to H$. As we have at least two edge-deleted subgraphs, this morphism is surjective. Thus $\Gamma$ is an epimorphism by Lemma \ref{EpiInj}. We apply the Fundamental Morphism Theorem to $\Gamma$:
\begin{equation} \label{eq:EdgeFMT} \xymatrix{R_\Gamma\ar[r]^-k & \displaystyle\coprod_{i\in I}G_i\times \coprod_{i\in I}G_i\ar@<.5ex>[r]^-{p_0}\ar@<-.5ex>[r]_-{p_1} & \displaystyle\coprod_{i\in I}G_i \ar[r]^-\Gamma\ar[d]_-q & H \ar@<.5ex>[r]^-{i_0}\ar@<-.5ex>[r]_-{i_1} & H+H \ar[r]^-{k^*} & R_\Gamma^*\\
 & & I\ar@{-->}[r]_-{\exists !h} & I^*\ar[u]_-{q^*} & & }
\end{equation}
where $k=eq(\Gamma p_0,\Gamma p_1),$ $q=coeq(p_0k,p_1k),$ $k^*=coeq(i_0\Gamma,i_1\Gamma),$ $q^*=eq(k^*i_0,k^*i_1)$, and $h:I\to I^*$ is the unique isomorphism that makes the diagram commute.\par
Using the construction (\ref{eq:hprimeprime}) in the proof of the Fundamental Morphism Theorem and as $\Gamma$ is an epimorphism with $\Gamma=q^* h^\prime$, $q^*:I^*\to H$ is an epimorphism. Then by Lemma \ref{balanced} as $q^*$ is a monomorphism (it is an equalizer), it is an isomorphism. Hence $I$ is isomorphic to $H$.\par
If there exists an epimorphism $\displaystyle \delta:\coprod_{i\in I} G_i\rightarrow G$ such that $\delta p_0 k = \delta p_1 k$, then the universal mapping property of $I$ would yield a unique morphism $\Delta:I\rightarrow G$ where $\delta=\Delta q$. Then $\Delta$ must be an epimorphism as $\delta$ is. As $|V(I)|=|V(H)|=|V(G)|$ and $|E(I)|=|E(H)|=|E(G)|$ both finite, $\Delta$ is a bijection. Therefore by Lemma \ref{EpiInj} and Lemma \ref{balanced} $\Delta$ is an isomorphism, and $G$ is isomorphic to $H$.\par
Conversely, if the Edge Reconstruction Conjecture holds, such a $\delta$ exists by appending the isomorphism from $G$ to $H$ and from $H$ to $I$ to the morphism $q$. Thus, we establish a categorical reformulation of the Edge Reconstruction Conjecture:
\begin{theorem}[Categorical equivalence to Edge-Reconstructable]
Given the construction in (\ref{eq:EdgeFMT}), $G$ is edge-reconstructable if and only if there exists an epimorphism $\displaystyle \delta:\coprod_{i\in I}G_i \to G$ such that $\delta p_0 k=\delta p_1 k$.
\end{theorem}

By following the ideas outlined in the above discussion, we may establish a similar reformulation for the  Vertex Reconstruction Conjecture.

\begin{conjecture}[Vertex Reconstruction Conjecture]
Let $G$ and $H$ be finite simple graphs with at least $3$ vertices and single vertex-deleted subgraphs $G_i$ and $H_i$ for $i\in I$ for some indexing set $I$, where there are isomorphisms $\gamma_i:G_i \to H_i$ for all $i\in I$, then $G$ is isomorphic to $H$.
\end{conjecture}

Naturally, these $G_i, H_i$ and $\gamma_i$ also gives rise to a $\displaystyle\Gamma:\coprod_{i\in I}G_i\to H$, which in turn allows us to once again apply the Fundamental Morphism Theorem.

\begin{equation} \label{eq:VertexFMT} \xymatrix{R_\Gamma\ar[r]^-k & \displaystyle\coprod_{i\in I}G_i\times \coprod_{i\in I}G_i\ar@<.5ex>[r]^-{p_0}\ar@<-.5ex>[r]_-{p_1} & \displaystyle\coprod_{i\in I}G_i \ar[r]^-\Gamma\ar[d]_-q & H \ar@<.5ex>[r]^-{i_0}\ar@<-.5ex>[r]_-{i_1} & H+H \ar[r]^-{k^*} & R_\Gamma^*\\
 & & I\ar@{-->}[r]_-{\exists !h} & I^*\ar[u]_-{q^*} & & }
\end{equation}

Which leads to a second reformulation of a classical reconstruction conjecture:

\begin{theorem}[Categorical equivalence to Vertex-Reconstructable]
Given the construction in (\ref{eq:VertexFMT}), $G$ is vertex-reconstructable if and only if there exists an epimorphism $\displaystyle \delta:\coprod_{i\in I}G_i \to G$ such that $\delta p_0 k=\delta p_1 k$.
\end{theorem}
\begin{proof}
We first note that $\Gamma$ is an epimorphism.  To see this, consider that for each $G_i$ there is a map $\gamma_i\kappa_i:G_i\to H$ via the canonical isomorphism $\gamma_i$ and inclusion $\kappa_i$.  So for any vertex $v\in H$, there is an $H_i$ such that $v\in V(\kappa_i(H_i))$.  So it follows that  $V(\Gamma(\displaystyle\coprod_{i\in I}G_i))=V(H)$.  Similarly, given $uv\in E(H)$, there is an $H_j$ such that $u,v\in V(\kappa_j(H_j))$ since each $H_j$ is isomorphic to a single vertex-deleted subgraph of $H$.  So $uv\in E(\kappa_j(H_j))$ and $E(\Gamma(\displaystyle\coprod_{i\in I}G_i))=E(H)$.  Thus $\Gamma$ is an epimorphism.

This allows us to use the same arguments in the establishment of the reformulation of the Edge Reconstruction Conjecture.

Using the construction (\ref{eq:hprimeprime}) in the proof of the Fundamental Morphism Theorem and as $\Gamma$ is an epimorphism with $\Gamma=q^* h^\prime$, $q^*:I^*\to H$ is an epimorphism. Then by Lemma \ref{balanced} as $q^*$ is a monomorphism (it is an equalizer), it is an isomorphism. Hence $I$ is isomorphic to $H$.\par
If there exists an epimorphism $\displaystyle \delta:\coprod_{i\in I} G_i\rightarrow G$ such that $\delta p_0 k = \delta p_1 k$, then the universal mapping property of $I$ would yield a unique morphism $\Delta:I\rightarrow G$ where $\delta=\Delta q$. Then $\Delta$ must be an epimorphism as $\delta$ is. As $|V(I)|=|V(H)|=|V(G)|$ and $|E(I)|=|E(H)|=|E(G)|$ both finite, $\Delta$ is a bijection. Therefore by Lemma \ref{EpiInj} and Lemma \ref{balanced} $\Delta$ is an isomorphism, and $G$ is isomorphic to $H$.\par
Conversely, if G is vertex reconstructable, then there is an isomorphism $\varphi:H\to G$.  The above discussion also gives us an isomorphism $\psi:I\to H$.  So by defining $\delta=\varphi\psi q$ we find a $\delta$ such that $\delta p_0 k=\delta p_1 k$, since $q=eq(p_0k, p_1k)$.

\end{proof}

\section{Conclusion}
The Fundamental Morphism Theorem is a categorical extension of the Noether Isomorphism Theorem(s), which has had far-reaching consequences in the theory of Groups, Rings, and other algebraic objects, especially regarding questions of isomorphisms and homomorphisms.  It is natural then to see that by extending this result to graph categories, that we see an immediate application to long standing conjectures regarding graph isomorphisms.  As many classical graph theoretic notions and associated problems are reformulated in terms of graph homomorphisms, the authors expect that application of the Fundamental Morphism Theorem will key part of understanding and solving these problems.

%\section*{\refname}
%\addcontentsline{toc}{section}{References}
\bibliographystyle{abbrv}
\bibliography{FMT}

\end{document}